\documentclass[12pt]{article}
\usepackage{array}
\usepackage{amsthm,amsmath,amssymb}
\usepackage{fullpage}

\usepackage{color}
\usepackage[normalem]{ulem}

\newtheorem{thm}{Theorem}[section]
\newtheorem{theorem}[thm]{Theorem}
\newtheorem{lemma}[thm]{Lemma}
\newtheorem{corollary}[thm]{Corollary}

\title{Interlacing families and the Hermitian spectral norm of digraphs}

\author{Gary Greaves\thanks{Supported by JSPS KAKENHI;
grant number: 26$\cdot$03903}\\
  Research Center for Pure \\and Applied Mathematics\\
  Tohoku University, Sendai, Japan\\
  {\tt grwgrvs@gmail.com}
\and
  Bojan Mohar\thanks{Supported in part by the Canada Research Chairs program,
  by an NSERC Discovery Grant (Canada),
  and by the Research Grant P1--0297 of ARRS (Slovenia).}~\thanks{On leave from:
  IMFM, Department of Mathematics, University of Ljubljana, Ljubljana, Slovenia.}\\
  Department of Mathematics\\
  Simon Fraser University\\
  Burnaby, BC~~V5A 1S6\\
  {\tt mohar@sfu.ca}
\and
  Suil O\thanks{Supported in part by NSERC}\\
  Department of Mathematics\\
  Simon Fraser University\\
  Burnaby, BC~~V5A 1S6\\
  {\tt osuilo@sfu.ca}
}

\begin{document}

\maketitle

\begin{abstract}
It is proved that for any finite connected graph $G$, there exists an orientation of $G$ such that the spectral radius of the corresponding Hermitian adjacency matrix is smaller or equal to the spectral radius of the universal cover of $G$ (with equality if and only if $G$ is a tree). This resolves a problem proposed by Mohar. The proof uses the method of interlacing families of polynomials that was developed by Marcus, Spielman, and Srivastava in their seminal work on the existence of infinite families of Ramanujan graphs.
\end{abstract}

\section{Introduction}
\label{sec:introduction}

While the eigenvalues of adjacency matrices of graphs have been very well-studied, results about the eigenvalues of their directed counterparts remain relatively sparse.
One of the reasons for this disparity in attention is because, unlike for (undirected) graphs, the adjacency matrix of a directed graph is not symmetric; hence it is more difficult to study and its spectrum exhibits a weaker relationship to the combinatorial properties of digraphs.
In order to circumvent this, Guo and Mohar~\cite{GM}, and independently Li and Liu~\cite{LL}, introduced the so-called \emph{Hermitian adjacency matrix} to study directed graphs.

Let $G = (V,E)$ be a (simple, undirected) graph.
We define an \emph{orientation} of $G$ to be a skew-symmetric map $\sigma : V \times V \to \{0, \pm 1 \}$ such that $\sigma(u,v) \ne 0$ if and only if $uv\in E$.
We denote by $G^\sigma$ the graph $G$ together with the orientation $\sigma$ and we call $G^\sigma$ an \emph{oriented graph}\footnote{We consider an edge $uv$ oriented from $u$ to $v$ if $\sigma(u,v)=1$}.
Following Guo and Mohar~\cite{GM} (see also \cite{LL,CaversEtAl}), we define the \emph{Hermitian adjacency matrix} $H(G^\sigma)$ of $G^\sigma$ to be the matrix with its $(u,v)$-entry equal to $i\sigma(u,v)$, where $i = \sqrt{-1}$ is the complex unit.
Since the matrix $H(G^\sigma)$ is Hermitian, it has real eigenvalues,
which we arrange in non-increasing order
 $\lambda_1 \geqslant \lambda_2\geqslant \cdots\geqslant \lambda_{|V|}$. The largest
eigenvalue
in absolute value, $\rho(G^\sigma) = \max\{\lambda_1,|\lambda_{|V|}|\}$, is called the \emph{spectral radius} of $G^\sigma$. The spectral radius of $G$ provides an upper bound on $\rho(G^\sigma)$ for any orientation $\sigma$, see \cite{GM,MoDHS}. 
While it is known which orientations attain this upper bound~\cite{MoDHS}, it is much more elusive to find orientations for which the spectral radius is small.
Orientations with small spectral radius gain their importance in relation to the notion of quasirandomness in digraphs (see Chung and Graham \cite{CG_QRtmt} and Griffiths \cite{Gr}). With this motivation in mind, Mohar \cite{MoDHS} asked what is the minimum spectral radius taken over all orientations of a given graph $G$. In this note
Mohar's question is answered completely by a surprising application of interlacing polynomials from the seminal work of Marcus, Spielman, and Srivastava \cite{MSS1}.
It is this relationship that makes us believe that orientations with minimal spectral radius may gain importance comparable to that of expanders and Ramanujan graphs (see~\cite{HLW}).

In their breakthrough paper~\cite{MSS1}, Marcus, Spielman, and Srivastava introduced the method of interlacing families to show that there exist infinite families of regular bipartite Ramanujan graphs of every degree greater than $2$.
In particular, they showed that characteristic polynomials of signed adjacency matrices of a graph form a so-called \emph{interlacing family}.
The advantage of having an interlacing family is that one is guaranteed that there is a member of the family whose largest root is at most the largest root of the sum of the polynomials of the interlacing family.

Godsil and Gutman~\cite{GG} showed that the average of the characteristic polynomials of signed adjacency matrices of $G$ is equal to the matching polynomial $\mu_G(x)$ of $G$ (see Section \ref{sect:2} for the definition of the matching polynomial).
This enabled Marcus et al.\ to deduce that, for any connected graph $G$, there exists a signed adjacency matrix of $G$ whose largest eigenvalue is at most the largest root of $\mu_G(x)$.

For a matrix $M$,  
write $\lambda_{1}(M)$
(resp. $\rho(M)$) to denote the largest eigenvalue (resp. spectral radius) of $M$, and for a polynomial $p=p(x)$, we let
$\rho(p)$ denote the largest absolute value of a root of $p(x)$.
We denote by $\rho(G)$ the spectral radius of the adjacency matrix of a graph $G$.
Our main result is the following theorem.

\begin{theorem}
\label{thm:match}
 Let $G$ be a graph and let $\mu_G$ be its matching polynomial. Then there exists an orientation $\sigma$ of $G$ such that 
 $\lambda_{1}(H(G^\sigma)) \leqslant \rho (\mu_G)$.
\end{theorem}

It is known (see \cite{GM} or \cite{LL}) that, for any orientation $\sigma$ of $G$, the spectrum of $H(G^\sigma)$ is symmetric about $0$. Hence $\rho(H(G^\sigma)) = \lambda_{1}(H(G^\sigma))$ and thus we have the following corollary.

\begin{corollary}\label{cor:match}
	Let $G$ be a graph.
	Then there exists an orientation $\sigma$ of $G$ such that $\rho(H(G^\sigma)) \leqslant \rho(\mu_G)$.
\end{corollary}

It is further known~\cite[Lemma 3.6]{MSS1} that $\rho(\mu_G)$ is bounded above by $\rho(U_G)$, where $U_G$ denotes the universal cover of $G$ (see \cite{MSS1} for a definition). Note that $U_G$ is an infinite tree (unless $G$ itself is a tree). As shown in \cite{Mohar}, the spectral radius of an infinite graph can be defined as the supremum of $\rho(G')$ taken over all finite subgraphs $G'$ of $G$ (see \cite{MoWo} for more details). This implies that $\rho(\mu_G) \leqslant \rho(U_G)$, with equality if and only if $G$ is a tree (see \cite{Gd} and \cite{Mohar}).
Further, if $G$ is a tree{\color{blue},} then it is straightforward to check that, for any orientation $\sigma$, the matrix $H(G^\sigma)$ is cospectral with the adjacency matrix of $G$, and the universal cover $U_G$ is $G$ itself.
Hence we also have the following corollary.

\begin{corollary}
\label{cor:unicov}
 Let $G$ be a connected graph.
 Then there exists an orientation $\sigma$ of $G$ such that $\rho(H(G^\sigma)) \leqslant \rho(U_G)$.
 Equality is attained if and only if $G$ is a tree.
\end{corollary}

There are graphs for which every orientation $\sigma$ satisfies that $\rho(H(G^\sigma)) < \rho(U_G)$. An example is $K_3$. This motivates the following result.

\begin{theorem}
\label{thm:lower bound}
 Let $G$ be a graph and let $\mu_G$ be its matching polynomial. Then there exists an orientation $\sigma$ of $G$ such that $\rho(H(G^\sigma)) \geqslant \rho(\mu_G)$.
\end{theorem}

The rest of the paper is devoted to the proof of Theorems \ref{thm:match} and \ref{thm:lower bound}.

\section{Random orientations and matchings}
\label{sect:2}

Let $G = (V,E)$ be a graph on $n$ vertices.
Define $\operatorname{Ori}(G)$ to be the set of all orientations of $G$.
The proof of Theorem~\ref{thm:match} is based on the proof presented in \cite{MSS1}.
Indeed, we follow the course of Marcus, Spielman, and Srivastava, with two parts.
First we show that the average of the characteristic polynomials of $H(G^\sigma)$ over all orientations $\sigma \in \operatorname{Ori}(G)$ is equal to the matching polynomial of $G$.
Then we show that the characteristic polynomials of $H(G^\sigma)$ (taken over all $\sigma \in \operatorname{Ori}(G)$) form an interlacing family as defined in Section~\ref{sect:3}.

An $l$-\emph{matching} in $G$ is an $l$-subset $M \in \binom{E}{l}$ such that no two edges in $M$ share a common vertex.
Let $m_l$ be the number of $l$-matchings of $G$ and set $m_0 = 1$.
The \emph{matching polynomial} of $G$ is defined as
\[
	\mu_G(x):=\sum_{j \geqslant 0}(-1)^j m_j x^{n-2j}.
\]
Since $\mu_G(x)$ can be written in the form $p(x^2)$ or $xp(x^2)$, the roots of $\mu_G(x)$ are symmetric about $0$.

\begin{lemma} \label{lem:matchpol}
	For any graph $G$, we have
	$\mathop{{}\mathbb{E}}_{\sigma \in \operatorname{Ori}(G)}\det(xI-H(G^\sigma))=\mu_G(x).$
\end{lemma}

\begin{proof}
Our proof follows the proof of \cite[Theorem 3.7]{MSS1}.
For notational convenience, given an orientation $\sigma \in \operatorname{Ori}(G)$, we let $H^\sigma_{u,v}$ denote the $(u,v)$-entry of $H(G^\sigma)$.
Consider the expansion of the determinant as a sum over permutations in $\mathfrak S(V) = \{\pi:V\to V\mid \pi \textrm{ is bijective}\}$:
$$
  \mathop{{}\mathbb{E}}_{\sigma \in \operatorname{Ori}(G)}\det(xI-H(G^\sigma)) =
  \mathop{{}\mathbb{E}}_{\sigma \in \operatorname{Ori}(G)} \left( \sum_{\pi \in \mathfrak S(V)} \operatorname{sgn} \pi \prod_{v \in V} \left(xI-H(G^\sigma)\right)_{v,\pi(v)} \right).
$$
The entries of the matrix $xI-H(G^\sigma)$ can be viewed as mutually independent random variables, except that $H^\sigma_{u,v}$ and $H^\sigma_{v,u}$ are just inverse of each other and $H^\sigma_{u,v}\, H^\sigma_{v,u} = 1$ whenever $uv\in E(G)$. Since
$\mathop{{}\mathbb{E}}_{\sigma \in \operatorname{Ori}(G)}
H^\sigma_{u,v} = 0$  for every $u\ne v$, we have that
$$
   \mathop{{}\mathbb{E}}_{\sigma \in \operatorname{Ori}(G)} \prod_{v \in V} \left(xI-H(G^\sigma)\right)_{v,\pi(v)} = 0
$$
whenever the permutation $\pi$ is not an involution (has a cycle of length at least 3). The same holds if $\pi$ is an involution and there is a vertex $v$ with $\pi(v)\ne v$, where $uv\notin E(G)$. The remaining set of permutations, $\mathfrak I(V)$, consists of all involutions $\pi$ in $\mathfrak S(V)$ such that all transpositions of $\pi$ correspond to edges of $G$. Let $\mathfrak I_l(V)$ denote all involutions in $\mathfrak I(V)$ with $l$ disjoint transpositions (and $n-2l$ fixed points). Clearly, permutations in $\mathfrak I_l(V)$ are in bijective correspondence with the $l$-matchings in $G$. Note that for $\pi\in \mathfrak I_l(V)$, we have $\operatorname{sgn} \pi = (-1)^l$ and
$$
   \mathop{{}\mathbb{E}}_{\sigma \in \operatorname{Ori}(G)} \prod_{v \in V} \left(xI-H(G^\sigma)\right)_{v,\pi(v)} = x^{n-2l}.
$$
The equations given above together with linearity of expectation
imply the following:
\begin{align*}
  \mathop{{}\mathbb{E}}_{\sigma \in \operatorname{Ori}(G)}\det(xI-H(G^\sigma)) &=
   \sum_{\pi \in \mathfrak I(V)} \operatorname{sgn} \pi \mathop{{}\mathbb{E}}_{\sigma \in \operatorname{Ori}(G)} \prod_{v \in V} \left(xI-H(G^\sigma)\right)_{v,\pi(v)} \\
    &= \sum_{l=0}^{\lfloor n/2\rfloor} \sum_{\pi \in \mathfrak I_l(V)} (-1)^l x^{n-2l} \\
    &= \mu_G(x),
\end{align*}
which is what we were to prove.
\end{proof}

\section{Interlacing polynomials}
\label{sect:3}

A univariate polynomial is {\it real-rooted} if all of its coefficients and roots are real.

A real-rooted polynomial $g(x)=a\prod_{j=1}^{n-1}(x-\alpha_j)$ ($a\ne0$) {\it interlaces} a real-rooted polynomial $f(x)=b\prod_{j=1}^{n}(x-\beta_j)$  ($b\ne0$) if
$$\beta_1 \leqslant \alpha_1 \leqslant \beta_2 \leqslant \alpha_2 \leqslant \cdots \leqslant \alpha_{n-1} \leqslant \beta_n.$$
Polynomials $f_1,\ldots,f_k$ have a {\it common interlacing} if there is a polynomial $g$ so that $g$ interlaces $f_j$ for each $j$.

Following \cite{MSS1}, we define the notion of an interlacing family of polynomials as follows. Let $S_1,\ldots,S_m$ be finite sets, and for every assignment $s_1,\ldots,s_m \in S_1\times \cdots \times S_m$, let $f_{s_1,\ldots,s_m}(x)$ be a real-rooted polynomial of degree $n$ with positive leading coefficient. For a partial assignment $(s_1,\ldots,s_k) \in S_1\times \cdots \times S_k$ with $k<m$, define $$f_{s_1,\ldots,s_k}=\sum_{s_{k+1}\in S_{k+1}, \ldots, s_m \in S_m} f_{s_1,\ldots,s_k,s_{k+1},\ldots,s_m} $$
as well as
$$f_{\emptyset}=\sum_{s_1\in S_1, \ldots, s_m \in S_m} f_{s_1,\ldots,s_m}.$$
The polynomials $\{f_{s_1,\ldots,s_m} \}$ form an {\it interlacing family} if for every $k=0,\ldots,m-1$ and all $(s_1,\ldots,s_k) \in S_1\times \cdots \times S_k$, the polynomials $\{f_{s_1,\ldots,s_k,t}\}_{t\in S_{k+1}}$ have a common interlacing.

By $\rho_1(p)$ we denote the largest root of a real-rooted polynomial $p(x)$.

\begin{lemma}[See Theorem 4.4 in~\cite{MSS1}]
\label{lem:large}
Let $S_i$ be finite sets for all $i\in [m]$, and let $\{f_{s_1,\ldots,s_m}\}$ be an interlacing family.
Then there exists $(s_1,\ldots,s_m) \in S_1 \times \cdots \times S_m$
so that $\rho_1(f_{s_1,\ldots,s_m}) \leqslant \rho_1(f_{\emptyset})$.
\end{lemma}

The above lemma is needed for proving Theorem \ref{thm:match}. We need the following counterpart in order to prove Theorem \ref{thm:lower bound}.

\begin{lemma}
\label{lem:interlacing lower bound}
Let $S_i$ be finite sets for all $i\in [m]$, and let $\{f_{s_1,\ldots,s_m}\}$ be an interlacing family.
Then there exists $(s_1,\ldots,s_m) \in S_1 \times \cdots \times S_m$ so that
$\rho_1(f_{s_1,\ldots,s_m}) \geqslant \rho_1(f_{\emptyset})$.
\end{lemma}

\begin{proof}
As the proof is essentially the same as the proof of Theorem 4.4 in~\cite{MSS1}, we only give a sketch.
The proof is by induction on $m$. Observe first that for each $s_1\in S_1$, the polynomials $g_{s_2,\ldots,s_m} = f_{s_1,\ldots,s_m}$ taken over all $(s_2,\ldots,s_m) \in S_2 \times \cdots \times S_m$ form an interlacing family. By the induction hypothesis, for each $s_1\in S_1$, there exists $(s_2,\ldots,s_m) \in S_2 \times \cdots \times S_m$ so that
$\rho_1(f_{s_1,\ldots,s_m}) \geqslant \rho_1(f_{s_1})$.
From the definition of an interlacing family (for $k=0$) we see that the polynomials $\{f_{s_1}\}_{s_1 \in S_{1}}$ have a common interlacing. This implies that their sum $f_\emptyset$ has its largest root smaller than or equal to the largest root of one of the polynomials $f_{s_1}$. For this $s_1$, the corresponding $(s_2,\ldots,s_m) \in S_2 \times \cdots \times S_m$ shows that
$\rho_1(f_{s_1,\ldots,s_m}) \geqslant \rho_1(f_{s_1}) \geqslant \rho_1(f_\emptyset)$
which gives the conclusion of the lemma.
\end{proof}

The following result is proved in \cite{MSS1} for real vector spaces and real positive semidefinite matrices, but it holds also when we consider the complex vector space $\mathbb{C}^n$ and Hermitian positive semidefinite matrices.

\begin{lemma}[See Theorem 6.6 in~\cite{MSS1}]
\label{MSSmodify}
If $a_1,\ldots,a_m,b_1,\ldots,b_m$ are vectors in $\mathbb{C}^n$, $D$ is a Hermitian positive semidefinite matrix, and $p_1,\ldots,p_m$ are real numbers in $[0,1]$,  
then the polynomial
$$\displaystyle \sum_{S \subseteq [m]} \left(\prod_{j \in S} p_j \right) \left(\prod_{j \notin S} (1-p_j) \right) \det \left(xI+D+\sum_{j \in S} a_ja_j^* + \sum_{j \notin S} b_jb_j^* \right)$$
has only real roots.
\end{lemma}

Note that $^*$ denotes the Hermitian transpose of the vector.

We will now show that the characteristic polynomials taken over all orientations of a graph $G$ form an interlacing family.
In order to parametrize the family, we enumerate the edges of $G$, writing $E(G)= \{u_1v_1,u_2v_2,\dots, u_mv_m\}$ and, for each edge $u_iv_i$, we choose one of its endvertices, say $u_i$. Then we let $S_i=\{-1,1\}$ for $1\leqslant i\leqslant m$. Now, the orientations
$\sigma\in \operatorname{Ori}(G)$ are in bijective correspondence with the $m$-tuples $(s_1,\dots,s_m)\in S_1\times \cdots \times S_m$ by the rule $s_i = \sigma(u_i,v_i)$ ($1\leqslant i\leqslant m$). Under this correspondence, we define 
$$\mathfrak f_{s_1,\dots,s_m}(x) := \det(xI-H(G^\sigma)).$$

\begin{theorem}\label{thm:interlacing}
The polynomials $\{\mathfrak f_{s_1,\ldots,s_m} \}$ form an interlacing family.
\end{theorem}

\begin{proof}
By Lemma 4.5 in~\cite{MSS1}, a family of polynomials of the same degree and with positive leading coefficients has a common interlacing if and only if every convex combination of polynomials from the family is real-rooted.
To prove interlacing, it thus suffices to show that for every
$k=0,\ldots,m-1$, for all $(s_1,\ldots,s_k) \in \{-1,1\}^k$, and
for every $\lambda\in [0,1]$, the polynomial $q = \lambda 
\mathfrak f_{s_1,\ldots,s_k,1} + (1-\lambda)
\mathfrak f_{s_1,\ldots,s_k,-1}$ has only real roots. This will be proved by applying Lemma \ref{MSSmodify}. Note that $q$ can be written as
$$q = \sum_{s_{k+2},\dots,s_m\in\{\pm1\}} \bigl(\lambda  
\mathfrak f_{s_1,\ldots,s_k,1,s_{k+2},\dots,s_m} + (1-\lambda)  
\mathfrak f_{s_1,\ldots,s_k,-1,s_{k+2},\dots,s_m}\bigr).$$
This sum can be expressed as the sum of characteristic polynomials which has the form from Lemma \ref{MSSmodify}, by taking the following values for the constants $p_j$ and vectors $a_j,b_j$ ($1\leqslant j \leqslant m$).

First, we define $p_j=\frac{1}{2}(s_j+1)$ for $1\leqslant j\leqslant k$, $p_{k+1}=\lambda$, and $p_j=\frac{1}{2}$ for $k+2\leqslant j\leqslant m$. For any $S\subseteq [m]$ and for $T=\{j \in [k] \mid s_j=-1 \}$, 
we have
$$\left(\prod_{j \in S} p_j \right) \left(\prod_{j \notin S} (1-p_j) \right) =
\left\{
  \begin{array}{ll}
    0, & \hbox{if $T\cap S\ne \emptyset$;} \\
    2^{-(m-k-1)} \lambda, & \hbox{if $T\cap S = \emptyset$ and $k+1 \in S$;} \\
    2^{-(m-k-1)} (1-\lambda), & \hbox{if $T\cap S = \emptyset$ and $k+1 \notin S$.}
  \end{array}
\right.
$$

Next, we define
$a_j = e_{u_j} - ie_{v_j}$ and $b_j = e_{u_j} + ie_{v_j}$.
If $S\subseteq [m]$ is the set of all indices $j$ for which $s_j=1$, then
$$\sum_{j \in S} a_ja_j^* + \sum_{j \notin S} b_jb_j^* = -H(G^\sigma)+D,$$
where $D$ is the diagonal matrix containing the degrees of vertices in $G$. The above equalities show that
\begin{eqnarray*}
q(x) &=& \sum_{s_{k+2},\dots,s_m\in\{\pm1\}} \bigl(\lambda  
\mathfrak f_{s_1,\ldots,s_k,1,s_{k+2},\dots,s_m}(x) + (1-\lambda)
\mathfrak f_{s_1,\ldots,s_k,-1,s_{k+2},\dots,s_m}(x)
\bigr) \nonumber \\
     &=& \displaystyle  
     2^{m-k-1}
     \sum_{S \subseteq [m]} \left(\prod_{j \in S} p_j \right) \left(\prod_{j \notin S} (1-p_j) \right) \det \left(xI-D+\sum_{j \in S} a_ja_j^* + \sum_{j \notin S} b_jb_j^* \right). \nonumber
\end{eqnarray*}
Let $\Delta$ be the maximum degree in $G$. Then $q(x)$ can be written as $2^{m-k-1}r(y)$, where $y=x-\Delta$ and
$$
    r(y) = \displaystyle \sum_{S \subseteq [m]} \left(\prod_{j \in S} p_j \right) \left(\prod_{j \notin S} (1-p_j) \right) \det \left(yI+(\Delta I-D)+\sum_{j \in S} a_ja_j^* + \sum_{j \notin S} b_jb_j^* \right)
$$
for which Lemma \ref{MSSmodify} applies. The conclusion is that $r(y)$ and hence also $q(x)$ is real-rooted.
This completes the proof.
\end{proof}

Now Theorem~\ref{thm:match} and Theorem~\ref{thm:lower bound} follow from Lemma~\ref{lem:matchpol} and Theorem~\ref{thm:interlacing} together with Lemma~\ref{lem:large} and Lemma~\ref{lem:interlacing lower bound} respectively.

\end{document}